\documentclass[11pt]{amsart}

\usepackage{amsmath, amsfonts, amssymb,amsthm}
\usepackage{amstext}
\usepackage{mathrsfs}

\usepackage{float,epsf,subfigure}
\usepackage[all,cmtip]{xy}
\usepackage{hyperref}
\usepackage{algorithm}
\usepackage{algorithmic}
\usepackage{mathtools}
\usepackage{float}
\usepackage{bm}
\theoremstyle{plain}
\newtheorem{theorem}{Theorem}[section]

\newtheorem{cor}[theorem]{Corollary}
\newtheorem{conj}[theorem]{Conjecture}

\newtheorem{def-thm}[theorem]{Definition-Theorem}
\newtheorem{lemma}[theorem]{Lemma}

\newtheorem{defi}[theorem]{Definition}

\newtheorem{theoremnn}{Theorem}

\theoremstyle{definition}

{

\begin{document}
\title[Algebraic degeneracy of holomorphic curves]{Algebraic degeneracy of holomorphic curves}
\author[X.J. Dong and P.C. Hu] { Xianjing Dong and Peichu Hu }

\address{School of Mathematics \\ China University of Mining and Technology \\ Xuzhou \\ 221116 \\ Jiangsu \\ P. R. China}
\email{xjdong05@126.com}
\address{Department of Mathematics \\ Shandong University  \\  Shandong, Jinan, 250100, P.R. China}
\email{pchu@sdu.edu.cn}


\date{}
\maketitle \thispagestyle{empty} \setcounter{page}{1}

\begin{abstract}
\noindent  We consider  the algebraic degeneracy of holomorphic curves from a point of  view of meromorphic vector fields.   Employing the notion of Jocabian sections introduced by W. Stoll, 
we establish   a 
Second Main Theorem type inequality.  As  consequences,   several algebraic degeneracy theorems are obtained for holomorphic curves into a  complex projective variety.
\end{abstract}

\vskip\baselineskip

\setlength\arraycolsep{2pt}

\section{Introduction}

Recall that a smooth complex projective variety $X$ is said to be of \emph{general type} if   $K_X$  (canonical line bundle) is \emph{pseudo ample} (cf. \cite{Lang86}). 
The well-known  Green-Griffiths conjecture (cf. \cite{GG, Lang86, S2002})
stipulates  that 
\begin{conj}\label{conj1} Let  $X$ be  a  smooth complex projective  variety of general type. 
Then $X$  contains  no  algebraically non-degenerate holomorphic curves. 
\end{conj}
A stronger version than Conjecture \ref{conj1} is   Green-Griffiths-Lang conjecture \cite{Lang86} which asserts  that 
if  $X$ is of general type,  then there is a proper algebraic subvariety $Z$ of $X$ such that
every  nonconstant holomorphic curve $f:\mathbb C\rightarrow X$ is contained in $Z.$ 
For more details about GG and GGL 
conjectures,  refer to Clemens \cite{Cle},  Demaily \cite{Dem1,Dem2}, Ein \cite{Ein1,Ein2}, Griffiths \cite{Gri},   Lu-Yau \cite{L-Y},   Siu \cite{Siu87, S2002},  
Siu-Yeung \cite{SY},   and see also  \cite{HY, Lu, Nadel, No,  ru}, etc.
 
In this paper,  we   consider  the algebraic degeneracy of holomorphic curves into complex projective varieties in a viewpoint of   meromorphic vector fields.
 The  technical line  is to establish a Second Main Theorem type inequality of Nevanlinna theory.

To begin with, we  introduce the main results of the paper, some notations will be illustrated later.  Let $X$ be a smooth complex projective variety with a holomorphic line bundle $E$ over $X.$
 Consider  $f:\mathbb C^m\rightarrow X,$ a  holomorphic mapping into $X$ with $\dim_{\mathbb C}X=n>m.$ Then, $f$ induces a tangent mapping $f_* : T_{\mathbb C^m}\rightarrow T_{X}$
 between holomorphic tangent bundles. The first main result of the paper  is  the following  Second Main Theorem type inequality
 \begin{theoremnn}[=Theorem \ref{mainthm}]\label{abc}
Let $L$ be an ample line bundle over $X.$
 Set
 $$\mathfrak Z_j=f_*\Big(\frac{\partial}{\partial z_j}\Big), \ \
 j=1,\cdots,m.$$
Let  $\mathfrak X_1,$ $\cdots,\mathfrak X_{n-m}$ be linearly almost-independent meromorphic vector fields over $X,$ and let $\mathfrak X_1\wedge \cdots\wedge \mathfrak X_{n-m}$ be
of  pole order at most $E,$ i.e.,  there exists a nonzero  holomorphic section
$t: X\rightarrow E$ such that $t\otimes \mathfrak X_1\wedge\cdots \wedge \mathfrak X_{n-m}$ is holomorphic.
Assume  that 
  $\mathfrak Z_{1f}\wedge\cdots\wedge \mathfrak Z_{mf}\wedge \varphi\not\equiv 0$
with  $\varphi=(t\otimes \mathfrak X_1\wedge\cdots \wedge \mathfrak X_{n-m})_f.$
If $f(\mathbb C^m)\nsubseteq F^{-1}_\varphi(0)$, then
\begin{equation*}
T_f(r,K_X)+ N_{f, {\rm Ram}}(r)\leq_{\rm exc} T_f(r,E)+ O\Big(\log^+T_f(r,L)+\log^+\log r\Big).
\end{equation*}
\end{theoremnn}
Theorem \ref{abc} yields  the following  degeneracy result
\begin{theoremnn}[=Corollary \ref{Cor1}] 
 Set
 $$\mathfrak Z_j=f_*\Big(\frac{\partial}{\partial z_j}\Big), \ \
 j=1,\cdots,m.$$
Let  $\mathfrak X_1,$ $\cdots,\mathfrak X_{n-m}$ be linearly almost-independent meromorphic vector fields over $X,$ and let $\mathfrak X_1\wedge \cdots\wedge \mathfrak X_{n-m}$ be
of  pole order at most $E.$ If $E^*\otimes K_X$ is pseudo ample,  then every 
holomorphic mapping $f:\mathbb C^m\rightarrow X$  is algebraically degenerate or  satisfied with the equation
  $$\mathfrak Z_{1f}\wedge\cdots\wedge \mathfrak Z_{mf}\wedge (\mathfrak X_{1}\wedge\cdots\wedge \mathfrak X_{n-m})_f= 0.$$
 \end{theoremnn}
To more precisely, if $\dim_{\mathbb C} X=2,$  then we obtain 
 \begin{theoremnn}[=Theorem \ref{ddc}]\label{tm3}
Assume that  $\dim_{\mathbb C}X=2.$ Let $\mathfrak X$ be a nonzero meromorphic vector field of  pole order at most $E$  over $X.$  If $E^*\otimes K_X$ is pseudo ample, 
 then $X$  contains  no  algebraically non-degenerate holomorphic curves. 
\end{theoremnn}

Finally, we  generalize a theorem of  Siu \cite{S2002} (or see Nadel \cite{Nadel}). 
Let $\mathscr D$ be a meromorphic connection on $T_X$ with
pole order at most $E,$ i.e., $t\otimes\mathscr D$ is holomorphic for a nonzero section $t\in
H^0(X,E).$ We obtain 
 \begin{theoremnn}[=Theorem \ref{Siu-thm-Nst21}]\label{sxx}
Let $f:\mathbb{C}\rightarrow X$ be a  
holomorphic curve.
 If $E^{-n(n-1)/2}\otimes K_X$ is pseudo ample, 
then  the image of $f$ is  contained in the pole divisor $(t)$
of  $\mathscr D.$ 
\end{theoremnn}
 
\section{Jacobian sections}

We   will generalize the notion of  \emph{Jacobian sections} introduced by Stoll (cf. \cite{HY, Stoll}),
which  is useful to the proofs of our 
 main theorems.
In this section,  $M,X$ denote   complex manifolds of complex dimension $m, n$ respectively. Let 
$S$ be an analytic subset of $M,$ we say that $S$ is \emph{thin} if $M\setminus S$ is dense in $M.$ 

\subsection{Jacobian bundles}~

  Let
$$f:M\rightarrow X$$
 be a
holomorphic mapping.
Given a holomorphic vector bundle $\pi:E\rightarrow X,$   the {\it
pull-back bundle}\index{pullback bundle}
$\tilde{\pi}:f^*E\rightarrow M$ is naturally defined up to an
isomorphism $\tilde f: f^*E_x\rightarrow E_{f(x)}$ for every $x\in M$
such that $f\circ\tilde{\pi}=\pi \circ\tilde{f}.$
Or equivalently,
 $$ f^*E=\Big\{(x,\xi)\in M\times E: \ f(x)=\pi(\xi)\Big\} $$
satisfying  that $\tilde{\pi}(x,\xi)=x$ and $\tilde{f}(x,\xi)=\xi.$  Let $H^0(X,E)$ stand for the space of all holomorphic sections of $E$ over $X$. For $s\in H^0(X,E)$,  a {\it lifted
section}\index{lifted section} $s_f\in H^0(M, f^*E)$ of $s$
by $f$ is  defined  by
\begin{equation}\label{lifted section}
 s_f(x)=\tilde{f}^{-1}\big(s\circ f(x)\big), \ \  ^\forall x\in M.
 \end{equation}
 If $s\not\equiv0$ with  $f(M)\not\subseteq s^{-1}(0),$  then  $s_f\not\equiv 0.$ So,  the \emph{pull-back
divisor} $f^*(s)$ of $(s)$ exists, and the multiplicity functions  of divisors $(s_f)$ and $f^*(s)$ satisfy
 $$\mu_{(s_f)}\geq \mu_{f^*(s)}.$$

The {\it Jacobian bundle}\index{Jacobian bundle} of $f$  is defined by
 $$K(f)=K_M\otimes f^*K_X^*,$$
where
$K_X^*$ is the dual of   canonical line bundle $K_X$.  A holomorphic section $F$ of $K(f)$
over $M$ is called a {\it Jacobian section of $f$}.  The section $F$ is said to be  {\it
effective}\index{effective} if $F^{-1}(0)$ is thin. The zero
divisor $(F)$   is called the {\it ramification divisor of $f$ with respect to $F.$}

Given  a Jacobian section $F$ of $f.$   Now we fix a  Hermitian inner product
on $K_X^*\oplus K_X,$ the inner product can
  pull  back to
 $f^*(K_X^*)\oplus f^*(K_X).$
In further, it induces  a $K_M$-valued Hermitian inner product
 $$\langle \cdot ,\cdot\rangle: \ K(f)\oplus f^*K_X\rightarrow K_M$$
 in a natural way.  Let $V$ be an open subset of $X$ such that $\tilde V:=f^{-1}(V)\not=\emptyset.$ The section $F$ defines a linear mapping (denoted  also by $F$)
 $$\ F: \  H^0(V,K_X)\rightarrow H^0(\tilde{V},K_M)$$
by $F[\Psi]=\langle F, \Psi_f\rangle$  for
$\Psi\in H^0(V,K_X)$.

Take an open subset $W$ of $X$ with $\tilde{W}:=f^{-1}(W)\not=\emptyset.$  Let  $(U; z_1,...,z_m),$   $(V; w_1,...,w_n)$
be holomorphic coordinate charts of $M, X$ respectively such that
$V\subseteq W$ and $f(U)\subseteq V.$  For $\Psi\in H^0(W,K_X),$ we write $\Psi|_V=\tilde {\Psi}dw$  and
$$F|_U=\tilde{F}dz\otimes d^*w_f,$$  where $\tilde{\Psi}, \tilde{F}$ are holomorphic functions on $V, U$ respectively,
and 
 $$dz=dz_1\wedge\cdots\wedge dz_m,\quad
 d^*w=\frac{\partial}{\partial w_1}\wedge\cdots\wedge \frac{\partial}{\partial
 w_n}.$$
By definition, it yields  that 
 $$F[\Psi]\big|_U=\tilde{\Psi}\circ f\cdot F[dw]=\tilde{\Psi}\circ f\cdot \tilde{F}dz.$$
 Let $A^{2n}(W)$ denote the set of smooth $2n$-forms on $W.$  We  
extend  $F$ linearly to
 $$F: \ A^{2n}(W)\rightarrow A^{2m}(\tilde{W})$$
in  a natural way: if $\Omega\in A^{2n}(W)$  is
 expressed  locally as
 $\Omega|_V=i_n\rho dw\wedge d\bar{w},$ where 
 $$i_p=\Big(\frac{\sqrt{-1}}{2\pi}\Big)^p(-1)^{\frac{p(p-1)}{2}} p!, \ \ ^\forall p\geq1,$$
 then 
$$ F[\Omega]\big|_U=i_m\rho\circ
 f\cdot F[dw]\wedge\overline{F[dw]}=i_m\rho\circ
 f\cdot |\tilde{F}|^2dz\wedge d\bar{z}.$$
Clearly,  we have
$$F[g_1\Omega_1+g_2\Omega_2]=g_1\circ f\cdot F[\Omega_1]+g_2\circ f\cdot F[\Omega_2]$$
for  $g_1,g_2\in A^0(W),$  $\Omega_1,\Omega_2\in A^{2n}(W).$
Let $\kappa$ be  a Hermitian metric on $K(f),$ then 
$\|F\|^2_\kappa=\tilde\kappa|\tilde{F}|^2$ with $\tilde\kappa|_U=\|dz\otimes
d^*w_f\|^2_\kappa.$ There is  a form $\Theta\in A^{2m}(\tilde{W})$
such that
\begin{equation*}
F[\Omega]=\|F\|^2_\kappa\Theta
\end{equation*}
with
$$\Theta|_U=i_m\rho\circ
 f\cdot\tilde\kappa^{-1}dz\wedge d\bar{z}.$$
Obviously, $\Theta>0$ if and only if $\Omega>0$.

In what follows, we  give an extension of Jacobian sections. 
  \begin{defi}\label{def}
Let $E$ be a holomorphic line bundle over $X.$ A holomorphic section $F_E$ of $f^*E\otimes K(f)$
over $M$ is called  a {\it Jacobian
section of $f$ with respect to $E$.}  We say that $F_E$ is  {\it
effective}\index{effective} if $F^{-1}_E(0)$  is thin. The zero
divisor $(F_E)$ is called the {\it ramification
divisor of $f$  with respect to $F_E$.}
  \end{defi}\label{def}
Let  $F_E$  be a Jacobian section of $f$ with respect to $E,$ we define   a   $f^*E\otimes K_M$-valued  interior product
 $$\langle \cdot,\cdot\rangle: \ \big(f^*E\otimes K(f)\big)\oplus f^*K_X\rightarrow f^*E\otimes K_M$$
induced  from  the natural    interior  product
on $K_X^*\oplus K_X.$
 Then $F_E$  defines  a linear mapping
 $$F_E: \ H^0(V, K_X)\rightarrow H^0(\tilde{V},f^*E\otimes K_M)$$
by $F_E[\Psi]=\langle F_E, \Psi_f\rangle$ for
$\Psi\in\Gamma(V, K_X)$. Let $\sigma$ be a local holomorphic frame of $E$ restricted to $V,$  and  write $F_E$ as
$$F_E|_U=\tilde{F}_E\sigma_f\otimes dz\otimes d^*w_f.$$  Let $\Psi\in H^0(W,K_X)$ with the  expression given  before,  then we have 
 $$F_E[\Psi]\big|_U=\tilde{\Psi}\circ f\cdot \sigma_f\otimes F_E[dw]=\tilde{\Psi}\circ f \cdot\tilde{F}_E\cdot\sigma_f\otimes dz.$$
Equip $E$ with  a Hermitian metric.  Similarly,  $F_E$  can  extend linearly to
 $$F_E: \ A^{2n}(W)\rightarrow A^{2m}(\tilde{W})$$
by  
\begin{equation}\label{qq}
 F_E[\Omega]\big|_U=i_m\rho\circ
 f\cdot \|\sigma_f\|^2\cdot|\tilde{F}_E|^2dz\wedge d\overline{z}
\end{equation}
for  $\Omega|_V=i_n\rho dw\wedge d\overline{w}\in A^{2n}(W).$
If $\kappa_E$ is  a Hermitian metric on  $f^*E\otimes K(f),$ then 
$\|F_E\|^2_{\kappa_E}=\tilde\kappa_E|\tilde{F}_E|^2$  with  $\tilde\kappa_E|_U=\|\sigma_f\otimes dz\otimes
d^*w_f\|^2.$ Also, there exists a form $\Theta\in A^{2m}(\tilde{W})$
such that
$$F_E[\Omega]=\|F_E\|^2_{\kappa_E}\Theta$$
with
$$\Theta|_U=i_m\rho\circ
 f\cdot\tilde\kappa_E^{-1}dz\wedge d\bar{z}.
$$
It is clear   that  $\Theta>0$ if and only if $\Omega>0$.

\subsection{Holomorphic fields} ~

 Let $E$  be a holomorphic line bundle over $X.$ Now we begin with  the notion of ``\emph{linearly almost-independent}" and ``\emph{pole order at most $E$}" for 
 vector fields.
  \begin{defi}\label{def}
 Let  $\mathfrak X_1,\cdots, \mathfrak X_k$  be meromorphic vector fields over $X.$
 We say that $\mathfrak X_1, \cdots, \mathfrak X_k$  are linearly almost-independent if 
 $\mathfrak X_1\wedge\cdots\wedge \mathfrak X_k\not\equiv0.$
We say that $\mathfrak X_1\wedge\cdots\wedge \mathfrak X_k$  is of   pole order at most $E$,  if  there exists a nonzero holomorphic section $t: X\rightarrow E$  such that
$t\otimes \mathfrak X_1\wedge\cdots\wedge \mathfrak X_k$
 is holomorphic.  For such $t,$ we have
$$
  \big(t\otimes \mathfrak X_1\wedge\cdots\wedge \mathfrak X_k\big)_f\in H^0\Big(M,f^*\big(E\otimes\bigwedge^{k}T_X\big)\Big).
$$
 \end{defi}
 Assume that $m<n$ and set
$q:=n-m.$ Apparently,  there exists a unique homomorphism
 $$\hat{f}: \ f^*\Big(E\otimes\bigwedge^{m}T^*_X\Big)\rightarrow
 f^*E\otimes\bigwedge^{m}T^*_M$$
 such that $\hat{f}(\xi_f)=f^*\xi$ for all $\xi\in
H^0\left(V,E\otimes\wedge^{m}T^*_X\right),$ where $V$ is  open  in $X$ with
$\tilde{V}:=f^{-1}(V)\not=\emptyset.$
Recall that  the interior
product
 $$\angle: \ K_X\oplus \Big(E\otimes \bigwedge^{q}T_X\Big)\rightarrow E\otimes \bigwedge^{m}
 T_X^*$$
is  defined by $\theta(\alpha\angle \xi)=(\xi\wedge\theta)(\alpha)\in E$ for every $\theta\in \bigwedge^{m}
 T_X,$ where $\alpha\in K_X$ and
 $\xi\in E\otimes \bigwedge^{q}T_N.$
 The interior product
  pulls back to an interior product
 $$\angle: \ f^*K_X\oplus f^*\Big(E\otimes\bigwedge^{q}T_X\Big)\rightarrow
 f^*\Big(E\otimes\bigwedge^{m}T^*_X\Big).$$
 
We call  $\varphi\in
H^0\left(M,f^*\left(E\otimes\bigwedge^{q}T_X\right)\right)$
 a {\it holomorphic field of
$f$ over $M$ of degree $q$ with respect to $E.$}  The section  $\varphi$ defines  a Jacobian
section $F_\varphi$ of $f$ with respect to $E$  by
 $$F_\varphi|_{\tilde{V}}=\hat{f}(\Psi_f\angle\varphi)\otimes
 \Psi_f^*,$$
 where $\Psi\in H^0(V,K_X)$ vanishes nowhere on $V,$  and $\Psi^*$ denotes
the dual  of $\Psi.$ Equivalently,  $F_\varphi$ is described by
 $F_\varphi[\Psi]= \hat{f}(\Psi_f\angle \varphi)$
  for every $\Psi\in H^0(V,K_X)$ with  $\tilde{V}\not=\emptyset$ for $V$ open in $X.$
 We say that $\varphi$ is 
  {\it effective}\index{effective} if $F_\varphi$ is
effective.

Let us  treat the existence of effective holomorphic fields of a holomorphic mapping $f: M\rightarrow X$ under the dimension condition $q:=n-m>0.$ Fix an integer $k$ such that $1\leq k\leq n,$ and 
 denote by $J_{1,k}^n$  be  the set of  all increasing injective mappings
\begin{equation*}
    \lambda: \ \mathbb{Z}[1,k]\longrightarrow \mathbb{Z}[1,n],
\end{equation*}
where $\mathbb{Z}[r,s]$ ($r\leq s$) denotes  the set of  integers $j$ with  $r\leq j\leq s$.
If $\lambda\in J_{1,k}^n,$ then $\lambda^{\bot}$ is uniquely defined such that $(\lambda, \lambda^{\bot})$ is a permutation of $\{1, \cdots,n\}.$ Clearly,
 $\bot: J_{1,k}^n\rightarrow  J_{1, n-k}^n$ is a bijective mapping.

 \begin{lemma}\label{lem0}  Assume $M$ is stein. Then a holomorphic field $\varphi$ of $f$ over $M$ of
 degree $q$ with respect to $E$ exists such that $\varphi$ is effective if and only if $f$ is differentiably non-degenerate.
\end{lemma}
\begin{proof}
 Let $S$  be the set of  all $x\in M$ such that
the rank of Jacobian  matrix of $f$ at $x$ is smaller than $m.$  Fix  $x_0\in M,$
we take local holomorphic coordinate
charts $(U;z_1,\cdots,z_m)$ of $x_0$ in $M,$ and $(V; w_1,\cdots,w_n)$ of
$f(x_0)$ in $X$ with $f(U)\subseteq V.$ We may assume that $E|_V\cong V\times\mathbb C.$
Set
$$d^*w_j=\frac{\partial}{\partial w_j}, \ \ j=1,\cdots, n.$$
For $\nu\in J^n_{1,m},$  we write
$f^*dw_\nu=A_\nu dz,$
where
 $dw_\nu=dw_{\nu(1)}\wedge\cdots\wedge d_{\nu(m)}$ and $dz=dz_1\wedge\cdots dz_m.$
By  definition of ranks,  it yields  that
$$S\cap U=\bigcap_{\nu\in J^n_{1,m}}A^{-1}_\nu(0).$$
\ \ \ \ \emph{Necessarity.} Assume  that $\varphi$ is an effective field of $f$ over $M$ of
 degree $q$ with respect to $E,$ we  prove that $f$ is differentiably non-degenerate.
 Let $\sigma$ be a local holomorphic
 frame of $E$ restricted to $V,$  we write
$$
\varphi|_U=\sum_{\lambda\in J^n_{1,q}}\varphi_\lambda \sigma_f\otimes d^*w_{\lambda f},
$$
where
$d^*w_{\lambda}=d^*w_{\lambda(1)}\wedge\cdots\wedge d^*w_{\lambda(q)}.$
 A simple computation shows that
$$dw_f\angle\varphi|_U=\sum_{\lambda\in J^n_{1,q}}{\rm sign}(\lambda^\bot, \lambda)
\varphi_\lambda \sigma_f\otimes dw_{\lambda^\bot f}.$$
Hence, we obtain
 $$F_\varphi[dw]\big|_U=
\hat{f}(dw_f\angle\varphi)=\sum_{\lambda\in J^n_{1,q}}{\rm sign}(\lambda^\bot, \lambda)
\varphi_\lambda A_{\lambda^\bot}\sigma_f\otimes dz,
$$
which implies that
$$(F_\varphi)\cap U=\Bigg(\sum_{\lambda\in J^n_{1,q}}{\rm sign}(\lambda^\bot, \lambda)
\varphi_\lambda A_{\lambda^\bot}\Bigg)\supseteq
S\cap U.$$
Since $\varphi$ is effective, then $S$ is thin. Thus, $f$ is differentiably non-degenerate.

\emph{Sufficiency.} Assume  that $f$ is differentiably non-degenerate,  we show that there exists an effective field of $f$ over $M$ of
 degree $q$ with respect to $E.$ 
For 
$x_0\in M\setminus S,$    there exists    $\iota\in J_{1,m}^n$  satisfying
that $A_\iota(x_0)\not=0$.
Since $M$ is stein,   then  by Corollary 5.6.3   in \cite{Ho} we note that  there exists   $\tau\in H^0(M, f^*E)$
and $s_j\in H^0(M, f^*T_X)$  such that
$$\tau(x_0)\not=0; \ \ \ s_j(x_0)=d^*w_{jf}(x_0), \ \ j=1,\cdots,n.$$
Set $s_{\iota}=s_{\iota(1)}\wedge\cdots\wedge s_{\iota(m)},$  then 
$$\tau\otimes s_{\iota^{\bot}}\in H^0\Big(M,f^*E\otimes\bigwedge^{q}{f^*
T_X}\Big).$$
Now define $\varphi:=\tau\otimes s_{\iota^{\bot}},$
 which is a holomorphic
field of $f$ over $M$ of degree $q$ with respect to $E.$ Next, we prove  that $\varphi$ is effective.  Write
$$\varphi|_U=\sum_{\lambda\in J^n_{1,q}}\varphi_\lambda \tau\otimes (d^*w_{\lambda})_f,$$
which is satisfied with 
$$\varphi_\lambda(x_0)=1, \ \ \lambda=\iota^{\bot}; \ \ \ \varphi_\lambda(x_0)=0, \ \ \lambda\not=\iota^{\bot}.$$
So, it yields  that 
$$F_\varphi[dw]=\sum_{\lambda\in J^n_{1,q}}{\rm sign}(\lambda^\bot, \lambda)
\varphi_\lambda A_{\lambda^\bot}\tau\otimes dz
$$
with
$$F_\varphi[dw](x_0)={\rm sign}(\iota, \iota^\bot)A_{\iota}(x_0)\tau(x_0)\otimes dz(x_0)\not=0.
$$
Notice that
$$(F_\varphi[dw])\cap U=(F_\varphi)\cap U,$$
hence $x_0$ is not a zero of $F_\varphi,$ which implies that $ {\rm Supp}(F_\varphi)\subseteq S$ is thin since $S$ is thin. Therefore, $\varphi$ is  effective. We conclude the proof.
 \end{proof}

\begin{theorem}
Assume   $f$ is differentiablely non-degenerate.  Let $\mathfrak X_1,\cdots, \mathfrak X_n$ be linearly almost-independent
 meromorphic vector fields  over $X,$ and let $\mathfrak X_1\wedge\cdots\wedge \mathfrak X_n$ be of pole order at most $E,$ i.e.,  there is
 a nonzero holomorphic section $t: X\rightarrow E$ such that $t\otimes \mathscr X$ is holomorphic with 
$\mathscr X:=\mathfrak X_1\wedge\cdots \wedge \mathfrak X_n.$
 If $f(M)\nsubseteq {\rm Supp}(t\otimes \mathscr X),$
then there is
$\lambda\in J_{1,q}^n$ such that
 $(t\otimes \mathfrak X_{\lambda(1)}\wedge\cdots\wedge
 \mathfrak X_{\lambda(q)})_f$
is an effective holomorphic field of $f$ over $M$ of degree $q$ with respect to $E$.
\end{theorem}
\begin{proof}
The conditions imply that  the lifted section
$$t_f\otimes\mathscr X_{f}\in H^0(M,f^*(E\otimes K_X^*))$$
   exists with
$t_f\otimes\mathscr X_{f}\not\equiv 0$.  Let $S$  be the set of  all $x\in M$ such that
the rank of Jacobian  matrix of $f$ at $x$ is smaller than $m.$
 Note  that   $S, (\mathscr X_t), (t_f\otimes\mathscr X_{f})$ are  thin.
Take $x_0\in
M\setminus S$ such that $f(x_0)\not\in {\rm Supp}(t\otimes\mathscr X),$
then there exists local holomorphic coordinate
charts $(U;z_1,\cdots,z_m)$ of $x_0$ and $(V; w_1,\cdots,w_n)$ of
$f(x_0)$ with $f(U)\subseteq V,$ such that  $E|_V\cong V\times\mathbb C$ and $t\otimes \mathfrak X_1\wedge \cdots\wedge\otimes \mathfrak X_n\not=0$ on $V.$ 
Write $t=\tilde t\sigma,$ where  $\sigma$ is a local holomorphic
 frame of $E$ restricted to $V.$ One can factorize $\tilde t$ as $\tilde t=\tilde t_1\cdots\tilde t_n$ such that $\tilde t_1\mathfrak X_1,\cdots,\tilde t_n\mathfrak X_n$ are holomorphic  and have no zeros on $V,$ and hence 
form a local holomorphic frame of $T_X$ on $V.$ 
Let $\omega_1,\cdots, \omega_n$ denote   
  the dual frame relative to  $ \tilde t_1\mathfrak X_1,\cdots, \tilde t_n\mathfrak X_n$ restricted to $V.$
For $\nu\in J_{1,m}^{n},$  we write
\begin{equation*}
f^*\omega_\nu|_U=B_\nu dz,
\end{equation*}
 where
$\omega_\nu=\omega_{\nu(1)}\wedge\cdots\wedge
\omega_{\nu(m)}.$
 Then there exists  $\iota\in J_{1,m}^n$  such
that $B_\iota(x_0)\not=0$. Take  $\lambda=\iota^\bot\in J_{1,q}^n$ and 
  define
 $$\varphi:=\big(t\otimes \mathfrak X_{\lambda(1)}\wedge\cdots\wedge
 \mathfrak X_{\lambda(q)}\big)_f.$$
 Clearly, 
$\varphi$ is a holomorphic field of $f$ over $M$ of degree $q$ with respect to $E.$
Next, we show  that $\varphi$ is effective.  Set  $$\Psi=\omega_1\wedge\cdots\wedge \omega_n.$$
We have  
 $$\Psi_f\angle\varphi={\rm sign}(\iota,\iota^\bot) t_f\otimes \omega_{\iota
 f}$$
with $\omega_{\iota f}=(\omega_{\iota (1)} \wedge\cdots\wedge\omega_{\iota (m)})_f,$ which yields that
\begin{equation*}
F_\varphi[\Psi]\big|_U= \hat{f}(\Psi_f\angle \varphi)={\rm
sign}(\iota,\iota^\bot) B_{\iota}t_{f}\otimes dz.
\end{equation*}
So, $x_0\in M\setminus{\rm Supp}(F_\varphi)$ due to $t_f|_U\not=0$ and
$$(F_\varphi[\Psi])\cap U=(F_\varphi)\cap U.$$
That is to say,  ${\rm Supp}(F_\varphi)\subseteq S\cup\mathscr (t_f\otimes \mathscr X_{f})$  is thin. Hence,
$\varphi$ is effective.
\end{proof}

\begin{theorem}\label{thm3}
Let   $M=\mathbb{C}^m.$ Set
 $$\mathfrak Z_j=f_*\Big(\frac{\partial}{\partial z_j}\Big), \ \
 j=1,\cdots,m.$$
Then a holomorphic field $\varphi$ of $f$ over
$\mathbb{C}^m$ of degree $q$ with respect to $E$ is effective  if and only if
 $$\mathscr Z:=\mathfrak Z_{1f}\wedge\cdots\wedge \mathfrak Z_{mf}\wedge \varphi\not\equiv 0.$$
Moreover, we have $(F_\varphi)=(\mathscr Z)$.
\end{theorem}

\begin{proof}
  Take a
holomorphic coordinate chart  $(V; w_1,\cdots,w_n)$
of $X.$ We may assume that $E|_V\cong V\times \mathbb C.$
 Write
 $$f_j=w_j\circ f, \ \
 j=1,\cdots,n; \ \ \  \mathfrak Z_i=\sum_{k=1}^n\frac{\partial f_k}{\partial
 z_i}d^*w_k,\ \  i=1,\cdots,m,$$
 where $d^*w_k=\partial/\partial w_k$ for $1\leq k\leq m.$
A simple computation gives that
 $$\mathfrak Z_{1f}\wedge\cdots\wedge \mathfrak Z_{mf}=\sum_{\nu\in
 J_{1,m}^{n}}A_\nu d^*w_{\nu f},$$
where 
 $$A_\nu=\det\Big(\frac{\partial f_{\nu(j)}}{\partial
 z_i}\Big), \ \ \ d^*w_{\nu}=d^*w_{\nu(1)}\wedge\cdots\wedge d^*w_{\nu(m)}.$$
 Let $\sigma$ be a local holomorphic
 frame of $E$ on  $V$ and write  
$$
\varphi|_{f^{-1}(V)}=\sum_{\lambda\in J^n_{1,q}}\varphi_\lambda \sigma_f\otimes d^*w_{\lambda f},
$$
 where  
$d^*w_{\lambda}=d^*w_{\lambda(1)}\wedge\cdots\wedge d^*w_{\lambda(q)}.$ Hence, we obtain
$$\mathscr Z|_{f^{-1}(V)}=\sum_{\lambda\in J^n_{1,q}}{\rm sign}(\lambda^{\bot}, \lambda)
\varphi_\lambda A_{\lambda^{\bot}}\sigma_f\otimes d^*w_f.$$
On the other hand, we have
 $$F_\varphi[dw]\big|_{f^{-1}(V)}=
 \sum_{\lambda\in J_{1,q}^{n}}{\rm sign}(\lambda^\perp,\lambda)\varphi_\lambda
 A_{\lambda^\perp}\sigma_f\otimes dz,$$
 where $dz=dz_1\wedge\cdots\wedge dz_m.$
So,  we conclude that  
 $$(F_\varphi)\cap f^{-1}(V)=(F_\varphi[dw])\cap f^{-1}(V)=(\mathscr Z)\cap f^{-1}(V).$$
 The proof is completed. 
\end{proof}
\begin{theorem}\label{lem2} Assume  $f$ is differentiably non-degenerate.  Let  $\varphi$ be a holomorphic field
of $f$ over $M$ of degree $q$ with respect to $E.$ Equip  $M, X, E$ with Hermitian metrics $\alpha,\omega, h$
 respectively. Define
 a non-negative function $g$ 
by
 $$F_\varphi[\omega^n]=g^2f^*\omega^m.$$
Then $g\leq \|\varphi\|,$  where $\|\cdot\|$ is induced from $\omega, h.$
\end{theorem}
\begin{proof}

For any  $x_0\in M,$
we can  take a  local holomorphic frame
$(U; \theta_1,\cdots,\theta_m)$ of $T^*_M$ around  $x_0$ and a local holomorphic frame
 $(V; \psi_1,\cdots,\psi_n)$  of  $T^*_X$ around  $f(x_0)$ with   $f(U)\subseteq V,$ such that
 $$\alpha|_{x_0}=\frac{\sqrt{-1}}{2\pi}
 \sum_{j=1}^m\theta_j\wedge
 \bar{\theta}_j, \ \ \omega|_{f(x_0)}=\frac{\sqrt{-1}}{2\pi}\sum_{j=1}^n\psi_j\wedge
 \bar{\psi}_j.$$
In addition, we may assume that $E|_V\cong V\times \mathbb C.$ Set
$$ f^*\omega^m=u\alpha^m; \ \ \   f^*\psi_\nu=A_{\nu}\theta_1\wedge\cdots\wedge\theta_m,\ \  ^\forall \nu\in J_{1,m}^{n}
$$
with 
 $\psi_\nu=\psi_{\nu(1)}\wedge
\psi_{\nu(2)}\wedge\cdots
 \wedge \psi_{\nu(m)}.$
It is trivial to confirm that
\begin{equation}\label{drt}
 u(x_0)=\sum_{\nu\in J_{1,m}^{n}}|A_\nu(x_0)|^2\not=0.
 \end{equation}
   Let  $\sigma$ be a local holomorphic frame of $E$ and  write  
   \begin{equation}\label{123aw}
  \varphi|_U=\sum_{\lambda\in J_{1,q}^{n}}\varphi_\lambda\sigma_f\otimes
 \psi^*_{\lambda f}
\end{equation}
 with
 $\psi^*_\lambda=\psi^*_{\lambda(1)}\wedge\cdots
 \wedge \psi^*_{\lambda(q)},$
 where  $\psi^*_1,\cdots,\psi^*_n$ is  the dual frame relative to
$\psi_1,\cdots,\psi_n.$
Set 
$\Psi=\psi_1\wedge\cdots\wedge\psi_n.$ Then
 $$\Psi_f\angle\varphi=\sum_{\lambda\in J_{1,q}^{n}}{\rm sign}(\lambda^\perp,\lambda)\varphi_\lambda
\sigma_f\otimes\psi_{\lambda^\perp f}.$$
Consequently,
$$ F_\varphi[\Psi]\big|_U=\hat{f}(\Psi_f\angle\varphi)=
 \sum_{\lambda\in J_{1,q}^{n}}{\rm sign}(\lambda^\perp,\lambda)\varphi_\lambda
 A_{\lambda^\perp}\sigma_f\otimes\theta_1\wedge\cdots\wedge\theta_m.
$$
This leads to
 $$g^2(x_0)u(x_0)=\Bigg| \sum_{\lambda\in J_{1,q}^{n}}{\rm sign}(\lambda^\perp,\lambda)\varphi_\lambda(x_0)
 A_{\lambda^\perp}(x_0)\Bigg|^2\|\sigma_f(x_0)\|^2.$$
Combine  (\ref{drt}) and (\ref{123aw}) with Schwarz's inequality,
 $$\Bigg| \sum_{\lambda\in J_{1,q}^{n}}{\rm sign}(\lambda^\perp,\lambda)\varphi_\lambda(x_0)
 A_{\lambda^\perp}(x_0)\Bigg|^2\|\sigma_f(x_0)\|^2\leq u(x_0)\|\varphi(x_0)\|^2$$
which follows that  $g(x_0)\leq \|\varphi(x_0)\|.$ This proves the theorem.
\end{proof}

\section{Second Main Theorem type inequalities}

\subsection{Preliminaries}~

For $z=(z_1,\cdots,z_m)\in\mathbb C^m,$  set $\|z\|^2=|z_1|^2+\cdots+|z_m|^2$ and 
$$\alpha=dd^c\|z\|^2,\ \ \ \gamma=d^c\log\|z\|^2\wedge \left(dd^c\log\|z\|^2\right)^{m-1}$$
with $d^c=\frac{\sqrt{-1}}{4\pi}(\bar\partial-\partial)$ and $dd^c=\frac{\sqrt{-1}}{2\pi}\partial\bar{\partial}.$
Let $$f:\mathbb C^m\rightarrow X$$  be a holomorphic mapping, where $X$  is a  smooth complex projective variety.
Given  an ample  line bundle  $L\rightarrow X$
  which carries  a Hermitian metric $h.$
 The \emph{characteristic function} of $f$ with respect to $L$ is defined by
 $$ T_f(r,L)=\int_1^r\frac{dt}{t^{2m-1}}\int_{\|z\|<t}f^*c_1(L,h)\wedge\alpha^{m-1},$$
 up to a bounded term for different metrics $h$ on $L.$   Since a holomorphic line bundle  on $X$
 can be written as the difference of two  ample  line bundles, then   definition of $T_{f}(r,L)$   extends to
 an arbitrary holomorphic line bundle.  Let $D$ be a  divisor on $X,$
   define   \emph{proximity function} of $f$ with respect to $D$ by
$$m_f(r,D)=\int_{\|z\|=r}\log\frac{1}{\|s_D\circ f\|}\gamma,$$
where  $s_D$ is the canonical section associated to $D.$  Write $s_D=\tilde{s}_De$ locally, where $e$ is a local holomorphic frame of $([D],h).$  We define   \emph{counting function}
 of $f$ with respect to $D$ by
$$N_f(r,D)=\int_1^r\frac{dt}{t^{2m-1}}\int_{\|z\|=t}dd^c\log|\tilde{s}_D\circ f|^2\wedge\alpha^{m-1}.$$
Assume  that $f(\mathbb C^m)\not\subseteq D,$ then the First Main Theorem (cf. \cite{ru}) says that
$$T_f(r,[D])=m_f(r,D)+N_f(r,D)+O(1).$$
\begin{lemma}[Borel Lemma, \cite{No}] Let $\phi$ be a monotone increasing function on $[0,\infty)$ such that $\phi(r_0)>1$ for some $r_0\geq0.$
Then for any $\delta>0,$ there exists a set $E_\delta\subset[0,\infty)$ of finite Lebegue measure such that
$$\phi'(r)\leq \phi(r)\log^{1+\delta}\phi(r)$$
holds for $r>0$ outside $E_\delta.$
\end{lemma}
\begin{proof} Since $\phi$ is monotone increasing,
$\phi'(r)$ exists for almost all $r\geq0.$
Set
$$S=\left\{r\geq0: \phi'(r)>\phi(r)\log^{1+\delta}\phi(r)\right\}.$$
Then
 \begin{eqnarray*}
\int_Sdr &\leq&\int_0^{r_0}dr+\int_{S\setminus[0,r_0]}dr
\leq r_0+\int_{r_0}^\infty\frac{\phi'(r)}{\phi(r)\log^{1+\delta}\phi(r)}dr <\infty.
 \end{eqnarray*}
\end{proof}
Let $\kappa$ be a non-negative, locally integrable real-valued function on $\mathbb C^m.$ Set
$$T_\kappa(r)=\int_1^r\frac{dt}{t^{2m-1}}\int_{\|z\|<t}\kappa\alpha^m.$$
Then $T_\kappa(r)$ is a monotone increasing function on $[1,\infty).$
\begin{lemma}[Calculus Lemma]\label{cal} Let $\kappa$ be a non-negative, locally integrable  real-valued function  on $\mathbb C^m$ such that  
$r^{2m-1}T'_\kappa(r)$  is monotone increasing in $r$ and  $r_0^{2m-1}T'_\kappa(r_0)>1$ for some $r_0\geq1.$
Then
$$\log^+\int_{\|z\|=r}\kappa\gamma\leq_{\rm exc}
O\Big(\log^+ T_\kappa(r)+\log^+\log r\Big).$$
\end{lemma}
\begin{proof} Denote by $dV$ the  Euclidean volume measure on $\mathbb C^m,$ and  by
$d\sigma_r$  the   volume measure (induced by $dV$) on $\{\|z\|=r\}$ for $r>0.$ A well-known fact says that
$$dV=\frac{\pi^m}{m!}\alpha^m, \ \ \  d\sigma_r=\omega_{2m-1}r^{2m-1}\gamma,$$
where $\omega_{2m-1}$ is the Euclidean volume of  unit sphere in $\mathbb R^{2m}.$ Therefore,
\begin{eqnarray*}
\frac{d}{dr}T_\kappa(r)&=&\frac{1}{r^{2m-1}}\int_{\|z\|<r}\kappa\alpha^m
= \frac{m!}{\pi^mr^{2m-1}}\int_{\|z\|<r}\kappa dV \\
&=& \frac{m!}{\pi^mr^{2m-1}}\int_0^rdt\int_{\|z\|=t}\kappa d\sigma_t \\
&=& \frac{m!\omega_{2m-1}}{\pi^mr^{2m-1}}\int_0^rdt\int_{\|z\|=t}\kappa \gamma,
 \end{eqnarray*}
which results in
$$\int_{\|z\|=r}\kappa \gamma=\frac{\pi^m}{m!\omega_{2m-1}r^{2m-1}}\frac{d}{dr}\Big(r^{2m-1}\frac{d}{dr}T_\kappa(r)\Big).$$
Indeed, we can use Borel's lemma to get
\begin{eqnarray*}
\frac{d}{dr}T_\kappa(r)&\leq_{\rm exc}& T_\kappa(r)\log^{1+\delta}T_\kappa(r),  \\
\frac{d}{dr}\Big(r^{2m-1}\frac{d}{dr}T_\kappa(r)\Big)&\leq_{\rm exc}&
\Big(r^{2m-1}\frac{d}{dr}T_\kappa(r)\Big)
\log^{1+\delta}\Big(r^{2m-1}\frac{d}{dr}T_\kappa(r) \Big).
 \end{eqnarray*}
Put  the above together, we conclude   that
$$\log^+\int_{\|z\|=r}\kappa\gamma\leq_{\rm exc}
O\Big(\log^+ T_\kappa(r)+\log^+\log r\Big).$$
\end{proof}
\begin{lemma}[Green-Jensen formula, \cite{No}]\label{GJ} Let $u\not\equiv-\infty$ be a plurisubharmonic function on $\mathbb C^m.$ Then for  any $0<s<r,$
$$\int_{\|z\|=r}u\gamma-\int_{\|z\|=s}u\gamma=2\int_s^r\frac{dt}{t^{2m-1}}\int_{\|z\|<t}dd^cu\wedge\alpha^{m-1}$$
holds in the sense of currents.
\end{lemma}

\subsection{Second Main Theorem type inequalities}~

Let $$f:\mathbb C^m\rightarrow X$$
 be a holomorphic mapping into a smooth complex projective variety $X$ with complex dimension $n>m.$
Let $E$ be  a holomorphic line bundle over $X.$  For $\varphi,$  an effective holomorphic field
of $f$ over $M$ of degree $q$ with respect to $E,$   
   the  {\it ramification
term} $N_{f, {\rm Ram}}(r)$ of $f$ with respect to $F_\varphi$ is defined by 
$$N_{f, {\rm Ram}}(r)=N(r, (F_\varphi)).
$$

\begin{theorem}\label{mainthm}
Let $E$ be a holomorphic line bundle over $X,$ and let $L$ be an ample line bundle over $X.$
  Let $f:\mathbb C^m\rightarrow X$ be a  holomorphic mapping.   Set
 $$\mathfrak Z_j=f_*\Big(\frac{\partial}{\partial z_j}\Big), \ \
 j=1,\cdots,m.$$
Let  $\mathfrak X_1,$ $\cdots, \mathfrak X_{n-m}$ be linearly almost-independent meromorphic vector fields over $X,$ and let $\mathfrak X_1\wedge \cdots\wedge \mathfrak X_{n-m}$ be
of pole order at most $E,$ i.e.,  there exists a nonzero holomorphic section
$t: X\rightarrow E$ such that $t\otimes \mathfrak X_1\wedge\cdots \wedge \mathfrak X_{n-m}$ is holomorphic.
Assume  that 
  $\mathfrak Z_{1f}\wedge\cdots\wedge \mathfrak Z_{mf}\wedge \varphi\not\equiv 0$
with  $\varphi=(t\otimes \mathfrak X_1\wedge\cdots \wedge \mathfrak X_{n-m})_f.$
If $f(\mathbb C^m)\nsubseteq F^{-1}_\varphi(0)$, then
\begin{equation*}
T_f(r,K_X)+ N_{f, {\rm Ram}}(r)\leq_{\rm exc} T_f(r,E)+ O\Big(\log T_f(r,L)+\log^+\log r\Big).
\end{equation*}
\end{theorem}

\begin{proof}  It is known  from Theorem \ref{thm3} that $\varphi$ is an effective holomorphic field
of $f$ over $M$ of degree $q$ with respect to $E.$  By Lemma \ref{lem0}, $f$ is differentiably non-degenerate. Since $L$ is ample, then there exists a Hermitian metric $h$ on $L$ such that 
$\omega:=c_1(L,h)>0.$
Define a non-negative function $\xi$ of  $C^\infty$-class on $\mathbb C^m\setminus F^{-1}_\varphi(0)$  by
\begin{equation}\label{xi}
    F_\varphi[\omega^n]=\xi\alpha^m.
\end{equation}
Let $h_E$ be a Hermitian metric on $E.$  Observe  (\ref{qq}),  one obtains
$${\rm{Ric}}( F_\varphi[\omega^n])=f^*{\rm{Ric}}(\omega^n)+f^*c_1(E, h_E).$$
This  yields  from Poincar\'e-Lelong formula that
\begin{equation*}
    dd^c\log \xi= f^*c_1(K_N,h_K) -f^*c_1(E, h_E)+(F_\varphi)
\end{equation*}
in the sense of currents, here $h_K$ is the Hermitian metric on $K_X$ induced  by $\omega.$
Integrating both sides of  the current  equation above,  we get
\begin{eqnarray*}
&& \int_1^r\frac{dt}{t^{2m-1}}\int_{\|z\|<t}dd^c\log \xi\wedge\alpha^{m-1} \\  &=&
 T_f(r,K_X) -T_f(r,E)+N_{f, {\rm Ram}}(r)+O(1).
\end{eqnarray*}
Apply Green-Jensen formula again, we arrive at
\begin{eqnarray}\label{5q}
\frac{1}{2}\int_{\|z\|=r}\log \xi\gamma   &=& T_f(r,K_X)-T_f(r,E)+N_{f, {\rm Ram}}(r)+O(1).
\end{eqnarray}
On the other hand,
set
 $$F_\varphi[\omega^n]=g^2f^*\omega^m.$$
Since $V$ is compact,  by  definition of $\varphi$ and Theorem \ref{lem2},  we have
 $g\leq \|\varphi\|<c$ for a positive constant $c.$
 Combined with (\ref{xi}), we obtain
 $$\xi=g^2\frac{f^*\omega^m}{\alpha^m}$$
 which yields  that
  $$\int_{\|z\|=r}\log\xi\gamma=\int_{\|z\|=r}\log g^2\gamma+\int_{\|z\|=r}\log\frac{f^*\omega^m}{\alpha^m}\gamma.$$
 The boundedness of $g$ means that
 $$\int_{\|z\|=r}\log g^2\gamma\leq O(1).$$
Now we claim: 
 \begin{equation}\label{claim}
 \int_{\|z\|=r}\log\frac{f^*\omega^m}{\alpha^m}\gamma\leq_{\rm exc} O\big(\log T_f(r,L)+\log^+\log r\big).
 \end{equation}
By this  with (\ref{5q}),  we  can confirm the theorem.
Set 
\begin{equation*}
  f^*\omega\wedge\alpha^{m-1}=\varrho\alpha^m.
\end{equation*}
For any  $z_0\in \mathbb C^m,$ one may pick  a holomorphic coordinate system $w_1,\cdots,w_n$ near $f(z_0)$ such that
\begin{equation*}
 \omega|_{f(z_0)}=\frac{\sqrt{-1}}{2\pi}\sum_{j=1}^n dw_j\wedge d\bar{w}_j.
\end{equation*}
Set $f_j=w_j\circ f$ for $1\leq j\leq n,$   then
\begin{equation*}
    f^*\omega\wedge\alpha^{m-1}|_{z_0}=\frac{(m-1)!}{2}\sum_{j=1}^n\big\|\nabla f_j\big\|^2\alpha^m,
\end{equation*}
where $\nabla$ is the gradient operator on $\mathbb C^m.$
That is,
$$    \varrho|_{z_0}=\frac{(m-1)!}{2}\sum_{j=1}^n\big\|\nabla f_j\big\|^2.
$$Thus, we have 
$$\frac{f^*\omega^m}{\alpha^m}\big|_{z_0}\leq B\varrho^m$$
holds for a  positive number $B.$
By the arbitrariness of $z_0,$ we receive
$$\log\frac{f^*\omega^m}{\alpha^m}\leq m\log \varrho+O(1)$$
 on $\mathbb C^m.$ Indeed, by Calculus Lemma (Lemma \ref{cal})
\begin{eqnarray*}
\int_{\|z\|=r}\log\varrho\gamma
&\leq& \log\int_{\|z\|=r}\varrho\gamma \\
&=& \log \int_{\|z\|=r}\frac{f^*\omega\wedge\alpha^{m-1}}{\alpha^m}\gamma
\\
&\leq_{\rm exc}& O\Big(\log\int_1^r\frac{dt}{t^{2m-1}}\int_{\|z\|<t}f^*\omega\wedge\alpha^{m-1}+\log^+\log r\Big) \\
   &=& O\big(\log T_f(r,L)+\log^+\log r\big).
\end{eqnarray*}
This conforms the claim (\ref{claim}).  We conclude the proof.
\end{proof}

\begin{lemma}[\cite{HY}]\label{SLang87lem5.2}
 Let
$H$ be a very ample line  bundle over $X$ and let $L$ be a pseudo
ample   line bundle over $X$. Then $$\limsup_{j\to +\infty}
\frac{1}{j^n}\dim H^0(X, L^j\otimes H^*)>0.$$
\end{lemma}

\begin{cor}\label{Cor1}
Let $E$ be a holomorphic line bundle over $X.$
 Set
 $$Z_j=f_*\Big(\frac{\partial}{\partial z_j}\Big), \ \
 j=1,\cdots,m.$$
Let $\mathfrak X_1,$ $\cdots, \mathfrak X_{n-m}$  be  linearly almost-independent meromorphic vector fields over $X,$
and let $\mathfrak X_1\wedge \cdots\wedge \mathfrak X_{n-m}$ be
of  pole order at most $E.$ If $E^*\otimes K_X$ is pseudo ample,  then every holomorphic mapping
 $f:\mathbb C^m\rightarrow X$  is algebraically degenerate or  satisfied with the equation
  $$\mathfrak Z_{1f}\wedge\cdots\wedge \mathfrak Z_{mf}\wedge (\mathfrak X_{1}\wedge\cdots\wedge \mathfrak X_{n-m})_f= 0.$$
 \end{cor}
 
\begin{proof} 
Since $N$ is projective algebraic, there exists a very ample line
bundle $H$ over $X$. By Lemma~\ref{SLang87lem5.2}, for $j$ large
there exists a non-trivial holomorphic section $s$ of $L^j\otimes
H^*$, where $L=E^*\otimes K_X$.  Since $X$ is compact, a Hermitian metric $\kappa$ on $L^j\otimes
H^*$ exists such that $\|s\|_\kappa\leq 1.$
Assume   $f(\mathbb C^m)\not\subseteq {\rm Supp}((s)),$ otherwise, the curve $f$ is algebraically degenerate.  By  First Main Theorem 
$$
N_f(r,(s))  \leq T_f(r,L^j\otimes H^*)+O(1)  =jT_f(r,L)-T_f(r,H)+O(1)
$$
which implies that 
\begin{equation}
 T_f(r,E^*\otimes K_X)\geq \frac{1}{j}T_f(r,H) +O(1),
\end{equation}
which contradicts with  
Theorem \ref{mainthm}.  
\end{proof}

 \begin{theorem}\label{ddc}
Let $X$ be a  $2$-dimensional smooth complex projective variety.  
Let $\mathfrak X$ be a nonzero meromorphic vector field of  pole order at most $E$  over $X.$  If $E^*\otimes K_X$ is pseudo ample, 
 then $X$  contains  no  algebraically non-degenerate holomorphic curves. 
\end{theorem}
\begin{proof} By the assumption, there is a nonzero  holomorphic section
$t: X\rightarrow E$ such that $t\otimes \mathfrak X$ is holomorphic. 
Note from Corllary \ref{Cor1} that 
$f$ is algebraically degenerate or  satisfied with the equation
\begin{equation}\label{s123}
\mathfrak Z_f\wedge \mathfrak X_f=0 \Longleftrightarrow \mathfrak Z_f\wedge (t\otimes\mathfrak X)_f=0.
\end{equation}
Let $\mathscr U$ be a finite open covering of $X$  such that every  element of $\mathscr U$ carries a holomorphic coordinate system. Take an element  $V\in\mathscr U$ with a  holomorphic coordinate system $(w_1,w_2)$ on $V.$ Locally, write 
$$t\otimes\mathfrak X|_V=A_1(w_1,w_2)\frac{\partial}{\partial w_{1}}+A_2(w_1,w_2)\frac{\partial}{\partial w_{2}},$$
 where $A_1,A_2$ are holomorphic in $w_1,w_2.$ If $f(\mathbb C)\cap V\not=\emptyset,$
then
$$\mathfrak Z|_{f(\mathbb C)\cap V}=f'_1\frac{\partial}{\partial w_{1}}+
f'_2\frac{\partial}{\partial w_{2}}$$
with $f_1=w_1\circ f$ and $f_2=w_2\circ f.$
Substitute the expressions of $t\otimes\mathfrak X$ and $\mathfrak Z$ into  (\ref{s123}),  it follows that
$$A_1(f_1,f_2)df_2+A_2(f_1, f_2)df_1=0.$$
Consider  the ODE
$$A_1(w_1,w_2)dw_2+A_2(w_1, w_2)dw_1=0.$$
Apply the integrating factor method in theory of ODEs,   there exists  $\Phi$ which  is holomorphic in $w_1,w_2,$ such that the equation is solved generally by
$$\Phi(w_1,w_2)=C,$$
where  $C$ is an arbitrary constant.  Thus,  we get
$$\Phi(f_1,f_2)=C_0$$
for a constant $C_0$ determined by $f.$
  This means  that $f|_{f^{-1}(V)}$ is algebraically degenerate. Since $\mathscr U$ is finite, then $f$ is algebraically degenerate.
\end{proof}

\subsection{Final notes}~

Let $X$ be a $n$-dimensional smooth complex projective  variety. 
Let $f:\mathbb{C}\rightarrow X$ be a nonconstant holomorphic
curve, the natural lifted curve  $f':\mathbb C\rightarrow T_X$ is the
derivative of $f.$  Take a local holomorphic coordinates $w=(w_1,...,w_n)$ of $X$ near
$f(z).$  For convenience,  we write  $f=(f_1,...,f_n),$  where $f_\alpha=w_\alpha\circ f.$  
Set $\partial_z=\partial /\partial
 z,$  $f'$ can be expressed as
 $$f'=\sum_{\alpha=1}^nf'_\alpha\partial_\alpha, \ \ f'_\alpha=\partial_zf_\alpha.$$  
\ \ \  \ Let $\mathscr D$ be a meromorphic connection on $T_X$ with
pole order at most $E,$ i.e., $t\otimes\mathscr D$ is holomorphic for a nonzero $t\in
H^0(X,E),$ where  $E$ is a holomorphic line bundle over $X.$ Equip a Hermitian metric on $E.$ For a
positive integer $k,$ define
 $$f^{(k+1)}=(t\otimes\mathscr D_{f'})^kf': \ \mathbb{C}\rightarrow E^k\otimes 
 T_X,$$
where $\mathscr D_{f'}$ is the covariant derivative with respect to the
connection $\mathscr D$ and  the direction of $f'$. We
have
 $$\mathscr D_{f'}f'=\sum_{\alpha=1}^n\mathscr D_{f'} f'_\alpha \partial_\alpha,$$
where
 $$\mathscr D_{f'} f'_\alpha=\sum_\beta f'_\beta \mathscr D_\beta f'_\alpha= \partial_zf'_\alpha+
 \sum_{\beta,\gamma}{\Gamma}_{\beta\gamma}^\alpha
 f'_\beta f'_\gamma.$$
Then, the holomorphic mapping
 $$f'\wedge\cdots\wedge
 f^{(n)}: \ \mathbb{C} \rightarrow E^{\frac{n(n-1)}{2}}\otimes K^*_X$$
is well defined.  In particular, we
obtain a holomorphic curve
$$
 f^{(2)}\wedge\cdots\wedge
 f^{(n)}: \ \mathbb{C} \rightarrow E^{\frac{n(n-1)}{2}}\otimes \bigwedge^{n-1}T_X.
$$The image of $f$ is called {\it autoparallel with respect to the
connection} $\mathscr D$ if
$$
  f'\wedge\cdots\wedge f^{(n)}\equiv 0.
$$

By means of meromorphic connections,  Siu \cite{Siu87}\index{Siu}  obtained  a defect relation.  A special case
of {\it Siu's Theorem}\index{Siu's Theorem} is presented by Nadel
\cite{Nadel}\index{Nadel} as follows
\begin{theorem}[Siu]\label{Siu-thm-Nst}
Let $f:\mathbb{C}\rightarrow X$ be a transcendental
holomorphic curve. If  $E^{-n(n-1)/2}\otimes K_X$ is ample, 
then  either $f(\mathbb{C})$ is  contained in the pole divisor $(t)$ 
of   $\mathscr D$ or $f(\mathbb{C})$ is 
autoparallel with respect to $\mathscr D$.
\end{theorem}
Nadel \cite{Nadel} remarked that the restriction ``\emph{transcedental}"  in Theorem \ref{Siu-thm-Nst} is not needed  since one may always replace $f(z)$ by $f(e^z).$ 
In what follows, we  extend Theorem \ref{Siu-thm-Nst}.  We first establish the following  Second Main Theorem type inequality: 
\begin{theorem}\label{mainthm2} 
Let $L$ be an ample line bundle over $X.$
Let $f:\mathbb{C}\rightarrow X$ be a 
holomorphic curve. 
Assume that  $f(\mathbb C)$ is not autoparallel with respect to  $\mathscr D.$
If $f(\mathbb C)\nsubseteq F^{-1}_\varphi(0)$ with  $\varphi= (f^{(2)}\wedge\cdots\wedge f^{(n)})_f,$  then
\begin{equation*}
T_f(r,K_X)+ N_{f, {\rm Ram}}(r)\leq_{\rm exc} \frac{n(n-1)}{2}T_f(r,E)+ O\Big(\log T_f(r,L)+\log r\Big).
\end{equation*}
\end{theorem}
\begin{proof} Since  $f(\mathbb C)$ is not autoparallel with respect to  $\mathscr D,$ then 
  $$ f'_f\wedge \big(f^{(2)}\wedge\cdots\wedge f^{(n)}\big)_f\not\equiv 0.$$
Equip $L$ with a Hermitian metric $h$ such that 
$\omega:=c_1(L,h)>0.$ Set $F_\varphi[\omega^n]=g^2f^*\omega.$ From  (6.3.17) in \cite{HY} (see also \cite{Hu1}), we note that 
 $$ \int_{|z|=r}\log g\frac{d\theta}{2\pi}\leq
 \int_{|z|=r}\log \|\varphi\|\frac{d\theta}{2\pi}\leq_{\rm exc} O\Big(\log T_f(r,L)+\log r\Big) .$$
So,  the proof   can be completed similarly to Theorem~\ref{mainthm}. 
 \end{proof}
 By Nadel's remark in above, Theorem~\ref{mainthm2}  yields immediately that 
 \begin{cor}\label{Siu-thm-Nst2}
Let $f:\mathbb{C}\rightarrow X$ be a  
holomorphic curve.
 If $E^{-n(n-1)/2}\otimes K_X$ is pseudo ample, 
then either $f(\mathbb{C})$ is  contained in the pole divisor $(t)$
of  $\mathscr D$ or $f(\mathbb{C})$ is 
autoparallel with respect to $\mathscr D$.
\end{cor}

In further, by the use of  techniques of  Siu \cite{Siu87} and Nadel \cite{Nadel},  the second possible statement  for the conclusion of Corollary \ref{Siu-thm-Nst2} can be also removed.
Namely, we have
 \begin{theorem}\label{Siu-thm-Nst21}
Let $f:\mathbb{C}\rightarrow X$ be a  
holomorphic curve.
 If $E^{-n(n-1)/2}\otimes K_X$ is pseudo ample, 
then  $f(\mathbb{C})$ is  contained in the pole divisor $(t)$
of  $\mathscr D.$ 
\end{theorem}

\vskip\baselineskip

\label{lastpage-01}
\end{document}